\documentclass[graybox]{svmult}

\usepackage{mathptmx}
\usepackage{helvet}
\usepackage{courier}
\usepackage{type1cm}

\usepackage{makeidx}

\usepackage{multicol}
\usepackage[bottom]{footmisc}

\usepackage{amsmath,amssymb}

\usepackage{cite}

\makeindex

\begin{document}

\title*{A symmetric quantum calculus\thanks{Submitted 26/Sept/2011; 
accepted in revised form 28/Dec/2011; to Proceedings of International 
Conference on Differential \& Difference Equations and Applications,
in honour of Professor Ravi P. Agarwal, to be published by Springer
in the series \emph{Proceedings in Mathematics} (PROM).}}

\author{Artur M. C. Brito da Cruz, Nat\'{a}lia Martins and Delfim F. M. Torres}

\authorrunning{A. M. C. Brito da Cruz, N. Martins and D. F. M. Torres}

\institute{Artur M. C. Brito da Cruz$^{1, 2}$\\
\email{artur.cruz@estsetubal.ips.pt}\\[0.2cm]
Nat\'{a}lia Martins$^{2}$\\
\email{natalia@ua.pt}\\[0.2cm]
Delfim F. M. Torres$^{2}$\\
\email{delfim@ua.pt}\\[0.3cm]
$^1$Escola Superior de Tecnologia de Set\'{u}bal, Estefanilha, 2910-761 Set\'{u}bal, Portugal\\[0.2cm]
$^2$Center for Research and Development in Mathematics and Applications\\
Department of Mathematics, University of Aveiro, 3810-193 Aveiro, Portugal}

\maketitle


\abstract{We introduce the $\alpha,\beta$-symmetric difference derivative and the
$\alpha,\beta$-symmetric N\"{o}rlund sum. The associated symmetric quantum calculus
is developed, which can be seen as a generalization
of the forward and backward $h$-calculus.}


\section{Introduction}

Quantum derivatives and integrals play a leading role
in the understanding of complex physical systems.
The subject has been under strong development
since the beginning of the 20th century
\cite{MyID:111,Ernst,Jackson,Kac,Milne}.
Roughly speaking, two approaches to quantum calculus are available.
The first considers the set of points of study to be the lattice
$q^{\mathbb{Z}}$ or $h\mathbb{Z}$ and is nowadays
part of the more general time scale calculus
\cite{Agarwal:surv,BohnerDEOTS,Malinowska};
the second uses the same formulas for the quantum
derivatives but the set of study is the set $\mathbb{R}$ of real numbers
\cite{MyID:188,withMiguel01,MalinowskaTorres}.
Here we take the second perspective.

Given a function $f$ and a positive real number $h$,
the $h$-derivative of $f$ is defined by the ratio
$\left(f\left(x+h\right)-f\left(x\right)\right)/h$.
When $h\rightarrow0$, one obtains the usual
derivative of the function $f$. The symmetric $h$-derivative is defined by
$\left(f\left(  x+h\right)  -f\left(  x-h\right)\right)/(2h)$,
which coincides with the standard symmetric derivative \cite{Thomsom}
when we let $h\rightarrow 0$.

We introduce the $\alpha,\beta$-symmetric difference derivative
and N\"{o}rlund sum, and then develop the associated calculus.
Such an $\alpha,\beta$-symmetric calculus gives
a generalization to (both forward and backward) quantum $h$-calculus.

The text is organized as follows. In Section~\ref{sec:2} we recall
the basic definitions of the quantum $h$-calculus, including
the N\"{o}rlund sum, \textrm{i.e.}, the inverse operation
of the $h$-derivative. Our results are then given in Section~\ref{sec:3}:
in \S\ref{sec:3.1} we define and prove the properties of
the $\alpha,\beta$-symmetric derivative;
in \S\ref{sec:3.2} we define the $\alpha,\beta$-symmetric N\"{o}rlund sum;
and \S\ref{sec:3.3} is dedicated to mean value theorems for the
$\alpha,\beta$-symmetric calculus: we prove $\alpha,\beta$-symmetric
versions of Fermat's theorem for stationary points,
Rolle's, Lagrange's, and Cauchy's mean value theorems.


\section{Preliminaries}
\label{sec:2}

In what follows we denote by $\left\vert I\right\vert $
the measure of the interval $I$.

\begin{definition}
Let $\alpha$ and $\beta$ be two positive real numbers,
$I\subseteq\mathbb{R}$ be an interval
with $\left\vert I\right\vert >\alpha$,
and $f:I\rightarrow\mathbb{R}$.
The $\alpha$-forward difference
operator $\Delta_{\alpha}$ is defined by
\[
\Delta_{\alpha}\left[  f\right]  \left(  t\right)  :=\frac{f\left(
t+\alpha\right)  -f\left(  t\right)  }{\alpha}
\]
for all $t\in I\backslash\left[\sup I-\alpha,\sup I\right]$,
in case $\sup I$ is finite, or, otherwise, for all $t \in I$.
Similarly, for $\left\vert I\right\vert >\beta$
the $\beta$-backward difference operator $\nabla_{\beta}$ is defined by
\[
\nabla_{\beta}\left[  f\right]  \left(  t\right)  :=\frac{f\left(  t\right)
-f\left(t-\beta\right)}{\beta}
\]
for all $t\in I\backslash\left[  \inf I,\inf I+\beta\right]$,
in case $\inf I$ is finite, or, otherwise, for all $t\in I$.
We call to $\Delta_{\alpha}\left[  f\right]$ the $\alpha$-forward
difference derivative of $f$ and to $\nabla_{\beta}\left[f\right]$
the $\beta$-backward difference derivative of $f$.
\end{definition}

\begin{definition}
\label{def:o1}
Let $I \subseteq \mathbb{R}$ be such that
$a,b\in I$ with $a<b$ and $\sup I=+\infty$.
For $f:I\rightarrow\mathbb{R}$ we define the N\"{o}rlund sum
(the $\alpha$-forward integral) of $f$ from $a$ to $b$ by
\[
\int_{a}^{b}f\left(  t\right)  \Delta_{\alpha}t=\int_{a}^{+\infty}f\left(
t\right)  \Delta_{\alpha}t-\int_{b}^{+\infty}f\left(  t\right)
\Delta_{\alpha}t,
\]
where
\[
\int_{x}^{+\infty}f\left(  t\right)  \Delta_{\alpha}t=\alpha\sum
_{k=0}^{+\infty}f\left(  x+k\alpha\right),
\]
provided the series converges at $x=a$ and $x=b$. In that case, $f$ is
said to be $\alpha$-forward integrable on $\left[  a,b\right]$. We say that $f$
is $\alpha$-forward integrable over $I$ if it is $\alpha$-forward integrable
for all $a,b\in I$.
\end{definition}

\begin{remark}
If $f:I\rightarrow\mathbb{R}$
is a function such that $\sup I<+\infty$,
then we can easily extend $f$ to $\tilde{f}:\tilde{I}\rightarrow\mathbb{R}$
with $\sup\tilde{I}=+\infty$ by letting
$\tilde{f}|_{I}=f$ and $\tilde{f}|_{\tilde{I}\backslash I}=0$.
\end{remark}

\begin{remark}
Definition~\ref{def:o1} is valid for any two real points $a,b$
and not only for points belonging to $\alpha\mathbb{Z}$.
This is in contrast with the theory of time scales \cite{Agarwal:surv,BohnerDEOTS}.
\end{remark}

Similarly, one can introduce the $\beta$-backward integral.

\begin{definition}
\label{beta}
Let $I$ be an interval of $\mathbb{R}$
such that $a,b\in I$ with $a<b$ and $\inf I=-\infty$. For $f:I\rightarrow\mathbb{R}$
we define the $\beta$-backward integral of $f$ from $a$ to $b$ by
\[
\int_{a}^{b}f\left(  t\right)  \nabla_{\beta}t=\int_{-\infty}^{b}f\left(
t\right)  \nabla_{\beta}t-\int_{-\infty}^{a}f\left(  t\right)  \nabla_{\beta}t,
\]
where
\[
\int_{-\infty}^{x}f\left(  t\right)  \nabla_{\beta}t
=\beta\sum_{k=0}^{+\infty} f\left(  x-k\beta\right),
\]
provided the series converges at $x=a$ and $x=b$. In that case, $f$ is
said to be $\beta$-backward integrable on $\left[  a,b\right]$. We say that $f$
is $\beta$-backward integrable over $I$ if it is $\beta$-backward integrable
for all $a,b\in I$.
\end{definition}

The $\beta$-backward N\"{o}rlund sum has similar results and properties
as the $\alpha$-forward N\"{o}rlund sum.


\section{Main Results}
\label{sec:3}

We begin by introducing in \S\ref{sec:3.1}
the $\alpha,\beta$-symmetric derivative;
in \S\ref{sec:3.2} we define
the $\alpha,\beta$-symmetric N\"{o}rlund sum;
while \S\ref{sec:3.3} is dedicated to mean value theorems
for the new $\alpha,\beta$-symmetric calculus.


\subsection{The $\alpha,\beta$-Symmetric Derivative}
\label{sec:3.1}

In what follows, $\alpha,\beta\in\mathbb{R}_{0}^{+}$
with at least one of them positive and $I$ is an interval such that
$\left\vert I\right\vert >\max\left\{  \alpha,\beta\right\}$.
We denote by $I_{\beta}^{\alpha}$ the set
\[
I_{\beta}^{\alpha}=\left\{
\begin{array}
[c]{ccc}
I\backslash\left(  \left[  \inf I,\inf I+\beta\right]  \cup\left[  \sup
I-\alpha,\sup I\right]  \right) & \text{if} & \inf I\neq-\infty \wedge \sup I\neq+\infty\\
I\backslash\left(  \left[  \inf I,\inf I+\beta\right]  \right)  & \text{if}
& \inf I\neq-\infty\wedge\sup I=+\infty\\
I\backslash\left(  \left[  \sup I-\alpha,\sup I\right]  \right)  & \text{if}
& \inf I=-\infty\wedge\sup I\neq+\infty\\
I & \text{if} & \inf I=-\infty\wedge\sup I=+\infty.
\end{array}
\right.
\]

\begin{definition}
\label{def:s:ab:dd}
The $\alpha,\beta$-symmetric difference derivative
of $f:I\rightarrow\mathbb{R}$ is given by
$$
D_{\alpha,\beta}\left[  f\right]  \left(  t\right)  =\frac{f\left(
t+\alpha\right)  -f\left(  t-\beta\right)}{\alpha+\beta}
$$
for all $t\in I_{\beta}^{\alpha}$.
\end{definition}

\begin{remark}
The $\alpha,\beta$-symmetric difference operator
is a generalization of both the $\alpha$-forward
and the $\beta$-backward difference operators.
Indeed, the $\alpha$-forward difference is obtained
for $\alpha>0$ and $\beta=0$; while for $\alpha=0$ and $\beta>0$
we obtain the $\beta$-backward difference operator.
\end{remark}

\begin{remark}
The classical symmetric derivative \cite{Thomsom}
is obtained by choosing $\beta = \alpha$ and taking the limit
$\alpha \rightarrow 0$. When $\alpha=\beta=h > 0$,  
the $\alpha,\beta$-symmetric difference operator
is called the $h$-symmetric derivative.
\end{remark}

\begin{remark}
\label{rem:lc:der}
If $\alpha,\beta\in\mathbb{R}^{+}$, then
$D_{\alpha,\beta}\left[  f\right]\left(t\right)
=\frac{\alpha}{\alpha+\beta}\Delta_{\alpha}\left[  f\right]\left(t\right)
+\frac{\beta}{\alpha+\beta}\nabla_{\beta}\left[  f\right]\left(t\right)$,
where $\Delta_{\alpha}$ and $\nabla_{\beta}$ are, respectively,
the $\alpha$-forward and the $\beta$-backward differences.
\end{remark}

The symmetric difference operator has the following properties:

\begin{theorem}
Let $f,g:I\rightarrow\mathbb{R}$ and $c,\lambda\in\mathbb{R}$.
For all $t\in I_{\beta}^{\alpha}$ one has:

\begin{enumerate}
\item $D_{\alpha,\beta}\left[  c\right]  \left(  t\right)  =0$;

\item $D_{\alpha,\beta}\left[  f+g\right]  \left(  t\right)=D_{\alpha,\beta
}\left[  f\right]  \left(  t\right)  +D_{\alpha,\beta}\left[  g\right]
\left(  t\right)  $;

\item $D_{\alpha,\beta}\left[  \lambda f\right]  \left(  t\right)  =\lambda
D_{\alpha,\beta}\left[  f\right]  \left(  t\right)  $;

\item $D_{\alpha,\beta}\left[  fg\right]  \left(  t\right)  =D_{\alpha,\beta
}\left[  f\right]  \left(  t\right)  g\left(  t+\alpha\right)  +f\left(
t-\beta\right)  D_{\alpha,\beta}\left[  g\right]  \left(  t\right)  $;

\item $D_{\alpha,\beta}\left[  fg\right]  \left(  t\right)  =D_{\alpha,\beta
}\left[  f\right]  \left(  t\right)  g\left(  t-\beta\right)  +f\left(
t+\alpha\right)  D_{\alpha,\beta}\left[  g\right]  \left(  t\right)  $;

\item $\displaystyle D_{\alpha,\beta}\left[  \frac{f}{g}\right]  \left(
t\right)  =\frac{D_{\alpha,\beta}\left[  f\right]  \left(  t\right)  g\left(
t-\beta\right)  -f\left(  t-\beta\right)  D_{\alpha,\beta}\left[  g\right]
\left(  t\right)  }{g\left(  t+\alpha\right)  g\left(  t-\beta\right)  }$
\newline provided $g\left(  t+\alpha\right)  g\left(  t-\beta\right)
\neq0$;

\item $\displaystyle D_{\alpha,\beta}\left[  \frac{f}{g}\right]  \left(
t\right)  =\frac{D_{\alpha,\beta}\left[  f\right]  \left(  t\right)  g\left(
t+\alpha\right)  -f\left(  t+\alpha\right)  D_{\alpha,\beta}\left[  g\right]
\left(  t\right)  }{g\left(  t+\alpha\right)  g\left(  t-\beta\right)  }$
\newline provided $g\left(  t+\alpha\right)  g\left(  t-\beta\right)
\neq0$.
\end{enumerate}
\end{theorem}

\begin{proof}
Property 1 is a trivial consequence of Definition~\ref{def:s:ab:dd}.
Properties 2, 3 and 4 follow by direct computations:
\begin{align*}
D_{\alpha,\beta}\left[  f+g\right]  \left(  t\right)   & =\frac{\left(
f+g\right)  \left(  t+\alpha\right)  -\left(  f+g\right)  \left(
t-\beta\right)  }{\alpha+\beta}\\
& =\frac{f\left(  t+\alpha\right)  -f\left(  t-\beta\right)  }{\alpha+\beta
}+\frac{g\left(  t+\alpha\right)  -g\left(  t-\beta\right)  }{\alpha+\beta}\\
& =D_{\alpha,\beta}\left[  f\right]  \left(  t\right)  +D_{\alpha,\beta
}\left[  g\right]  \left(  t\right);
\end{align*}
\begin{align*}
D_{\alpha,\beta}\left[  \lambda f\right]  \left(  t\right)   & =\frac{\left(
\lambda f\right)  \left(  t+\alpha\right)  -\left(  \lambda f\right)  \left(
t-\beta\right)  }{\alpha+\beta}\\
&=\lambda\frac{f\left(  t+\alpha\right)  -f\left(  t-\beta\right)}{\alpha+\beta}\\
& =\lambda D_{\alpha,\beta}\left[  f\right]  \left(  t\right);
\end{align*}
\begin{align*}
D_{\alpha,\beta}\left[  fg\right]  \left(  t\right)   & =\frac{\left(
fg\right)  \left(  t+\alpha\right)  -\left(  fg\right)  \left(  t-\beta
\right)  }{\alpha+\beta}\\
& =\frac{f\left(  t+\alpha\right)  g\left(  t+\alpha\right)  -f\left(
t-\beta\right)  g\left(  t-\beta\right)  }{\alpha+\beta}\\
& =\frac{f\left(  t+\alpha\right)  -f\left(  t-\beta\right)  }{\alpha+\beta
}g\left(  t+\alpha\right)
+\frac{g\left(  t+\alpha\right)  -g\left(  t-\beta\right)  }{\alpha+\beta
}f\left(  t-\beta\right) \\
& =D_{\alpha,\beta}\left[  f\right]  \left(  t\right)  g\left(  t+\alpha
\right)  +f\left(  t-\beta\right)  D_{\alpha,\beta}\left[  g\right]  \left(
t\right).
\end{align*}
Equality 5 is obtained from 4 interchanging the role of $f$ and $g$.
To prove 6 we begin by noting that
\begin{align*}
D_{\alpha,\beta}\left[  \frac{1}{g}\right]  \left(  t\right)   & =\frac
{\frac{1}{g}\left(  t+\alpha\right)  -\frac{1}{g}\left(  t-\beta\right)
}{\alpha+\beta}
=\frac{\frac{1}{g\left(  t+\alpha\right)  }-\frac{1}{g\left(  t-\beta
\right)  }}{\alpha+\beta}\\
& =\frac{g\left(  t-\beta\right)  -g\left(  t+\alpha\right)  }{\left(
\alpha+\beta\right)  g\left(  t+\alpha\right)  g\left(  t-\beta\right)}
=-\frac{D_{\alpha,\beta}\left[  g\right]  \left(  t\right)  }{g\left(
t+\alpha\right)  g\left(  t-\beta\right)  }\text{.}
\end{align*}
Hence,
\begin{align*}
D_{\alpha,\beta}\left[  \frac{f}{g}\right]  \left(  t\right)
&=D_{\alpha,\beta}\left[  f\frac{1}{g}\right]  \left(  t\right)
=D_{\alpha,\beta}\left[  f\right]  \left(  t\right)  \frac{1}{g}\left(
t+\alpha\right)  +f\left(  t-\beta\right)  D_{\alpha,\beta}\left[  \frac{1}
{g}\right]  \left(  t\right) \\
& =\frac{D_{\alpha,\beta}\left[  f\right]  \left(  t\right)  }{g\left(
t+\alpha\right)  }-f\left(  t-\beta\right)  \frac{D_{\alpha,\beta}\left[
g\right]  \left(  t\right)  }{g\left(  t+\alpha\right)  g\left(
t-\beta\right)  }\\
& =\frac{D_{\alpha,\beta}\left[  f\right]  \left(  t\right)  g\left(
t-\beta\right)  -f\left(  t-\beta\right)  D_{\alpha,\beta}\left[  g\right]
\left(  t\right)  }{g\left(  t+\alpha\right)  g\left(  t-\beta\right)  }.
\end{align*}
Equality 7 follows from simple calculations:
\begin{align*}
D_{\alpha,\beta}\left[  \frac{f}{g}\right]  \left(  t\right)
&=D_{\alpha,\beta}\left[  f\frac{1}{g}\right]  \left(  t\right)
=D_{\alpha,\beta}\left[  f\right]  \left(  t\right)  \frac{1}{g}\left(
t-\beta\right)  +f\left(  t+\alpha\right)  D_{\alpha,\beta}\left[  \frac{1}
{g}\right]  \left(  t\right) \\
& =\frac{D_{\alpha,\beta}\left[  f\right]  \left(  t\right)  }{g\left(
t-\beta\right)  }-f\left(  t+\alpha\right)  \frac{D_{\alpha,\beta}\left[
g\right]  \left(  t\right)  }{g\left(  t+\alpha\right)  g\left(
t-\beta\right)  }\\
& =\frac{D_{\alpha,\beta}\left[  f\right]  \left(  t\right)  g\left(
t+\alpha\right)  -f\left(  t+\alpha\right)  D_{\alpha,\beta}\left[  g\right]
\left(  t\right)  }{g\left(  t+\alpha\right)  g\left(  t-\beta\right)  }.
\end{align*}
\end{proof}


\subsection{The $\alpha,\beta$-Symmetric N\"{o}rlund Sum}
\label{sec:3.2}

Having in mind Remark~\ref{rem:lc:der},
we define the $\alpha,\beta$-symmetric integral
as a linear combination of the $\alpha$-forward and the
$\beta$-backward integrals.

\begin{definition}
Let $f:\mathbb{R}\rightarrow\mathbb{R}$ and $a,b\in\mathbb{R}$, $a<b$.
If $f$ is $\alpha$-forward and $\beta$-backward integrable
on $\left[  a,b\right]$, then we define the
$\alpha,\beta$-symmetric integral of $f$ from $a$ to $b$ by
\[
\int_{a}^{b}f\left(  t\right)  d_{\alpha,\beta}t=\frac{\alpha}{\alpha+\beta
}\int_{a}^{b}f\left(  t\right)  \Delta_{\alpha}t+\frac{\beta}{\alpha+\beta
}\int_{a}^{b}f\left(  t\right)  \nabla_{\beta}t\text{.}
\]
The function $f$ is $\alpha,\beta$-symmetric integrable if it is
$\alpha,\beta$-symmetric integrable for all $a,b\in\mathbb{R}$.
\end{definition}

\begin{remark}
Note that if $ \alpha\in\mathbb{R}^{+}$ and $\beta=0$,
then $\displaystyle\int_{a}^{b}f\left(  t\right)
d_{\alpha,\beta}t=\int_{a}^{b}f\left(  t\right)  \Delta_{\alpha}t$;
if $\alpha=0$ and $\beta\in\mathbb{R}^{+}$, then
$\displaystyle\int_{a}^{b}f\left(  t\right)  d_{\alpha,\beta}t
=\int_{a}^{b}f\left(  t\right)  \nabla_{\beta}t$.
\end{remark}

The properties of the $\alpha,\beta$-symmetric integral
follow from the corresponding $\alpha$-forward
and $\beta$-backward integral properties.
It should be noted, however, that the equality
$D_{\alpha,\beta}\left[s \mapsto  \int_{a}^{s}f\left(\tau\right) d_{\alpha,\beta}
\tau\right](t) =f\left(  t\right)$
is not always true in the $\alpha,\beta$-symmetric calculus,
despite both forward and backward integrals satisfy the
corresponding fundamental theorem of calculus. Indeed, let
$f\left(  t\right)
=\displaystyle
\begin{cases}
\frac{1}{2^{t}} & \text{ if } t\in\mathbb{N},\\
0 & \text{otherwise}.
\end{cases}$
Then, for a fixed $t \in \mathbb{N}$,
\begin{align*}
\int_{0}^{t}\frac{1}{2^{\tau}}d_{1,1}\tau & =\frac
{1}{2}\int_{0}^{t}\frac{1}{2^{\tau}}\Delta_{1}
\tau+\frac{1}{2}\int_{0}^{t}\frac{1}{2^{\tau}}
\nabla_{1}\tau\\
& =\frac{1}{2}\left(  \sum_{k=0}^{+\infty}f\left(  0+k\right)
-\sum_{k=0}^{+\infty}f\left(  t+k\right)  \right)
+\frac{1}{2}\left(  \sum_{k=0}^{+\infty}f\left(  t-k\right)
-\sum_{k=0}^{+\infty}f\left(  0-k\right)  \right)  \\
&=\frac{1}{2}\left(  1+\frac{1}{2}+\cdots+\frac{1}{2^{t-1}}\right)  +\frac
{1}{2}\left(  \frac{1}{2^{t}}+\frac{1}{2^{t-1}}+\cdots+\frac{1}{2}\right)  \\
&=\frac{1}{2}\frac{1-\frac{1}{2^{t}}}{1-\frac{1}{2}}+\frac{1}{4}\frac
{1-\frac{1}{2^{t}}}{1-\frac{1}{2}}=\frac{3}{2}\left(  1-\frac{1}{2^{t}}\right)
\end{align*}
and $\displaystyle
D_{1,1}\left[s \mapsto \int_{0}^{s}\frac{1}{2^{\tau}}
d_{1,1}\tau\right](t) =\frac{3}{2}D_{1,1}\left[ s \mapsto  1-\frac{1}{2^{s}}\right](t)
=-\frac{3}{2}\frac{\frac{1}{2^{t+1}}-\frac{1}{2^{t-1}}}{2}
=\frac{9}{2^{t+3}}$.
Therefore,
$\displaystyle D_{1,1}\left[s \mapsto \int_{0}^{s}\frac{1}{2^{\tau}}
d_{1,1}\tau\right](t) \neq \frac{1}{2^{t}}$.


\subsection{Mean Value Theorems}
\label{sec:3.3}

We begin by remarking that if $f$ assumes its local maximum at
$t_{0}$, then there exist $\alpha,\beta\in\mathbb{R}_{0}^{+}$
with at least one of them positive, such that
$f\left(  t_{0}+\alpha\right)  \leqslant f\left(  t_{0}\right)$
and
$f\left(  t_{0}\right)  \geqslant f\left(  t_{0}-\beta\right)$.
If $\alpha,\beta\in\mathbb{R}^{+}$, this means that
$\Delta_{\alpha}\left[  f\right]  \left(  t\right)  \leqslant 0$
and $\nabla_{\beta}\left[  f\right]  \left(  t\right)  \geqslant 0$.
Also, we have the corresponding result for a local mimimum. If $f$ assumes its
local minimum at $t_{0}$, then there exist $\alpha,\beta\in\mathbb{R}^{+}$ such that
$\Delta_{\alpha}\left[  f\right]  \left(  t\right)  \geqslant 0$ and
$\nabla_{\beta}\left[  f\right]  \left(  t\right)  \leqslant 0$.

\begin{theorem}[The $\alpha,\beta$-symmetric Fermat theorem for stationary points]
\label{Symmetric Fermat's Theorem}
Let $f:\left[  a,b\right]  \rightarrow\mathbb{R}$
be a continuous function. If $f$ assumes a local extremum at
$t_{0} \in\left]  a,b\right[$, then there exist
two positive real numbers $\alpha$ and $\beta$ such that
$D_{\alpha,\beta}\left[  f\right]  \left(  t_{0}\right)  =0$.
\end{theorem}

\begin{proof}
We prove the case where $f$ assumes a local maximum at $t_{0}$. Then
there exist $\alpha_{1},\beta_{1}\in\mathbb{R}^{+}$ such that
$\Delta_{\alpha_{1}}\left[  f\right]  \left(  t_{0}\right)  \leqslant 0$
and
$\nabla_{\beta_{1}}\left[  f\right]  \left(  t_{0}\right)  \geqslant 0$.
If $f\left(  t_{0}+\alpha_{1}\right)  =f\left(  t_{0}-\beta_{1}\right)$,
then $D_{\alpha_{1},\beta_{1}}\left[  f\right]  \left(  t_{0}\right)  =0$.
If $f\left(  t_{0}+\alpha_{1}\right)  \neq f\left(  t_{0}-\beta_{1}\right)$,
then let us choose $\gamma=\min\left\{  \alpha_{1},\beta_{1}\right\}$.
Suppose (without loss of generality) that $f\left(  t_{0}-\gamma\right)
>f\left(  t_{0}+\gamma\right)$. Then,
$f\left(  t_{0}\right)  >f\left(  t_{0}-\gamma\right)  >f\left(  t_{0}
+\gamma\right)$
and, since $f$ is continuous, by the intermediate value theorem there
exists $\rho$ such that $0<\rho<\gamma$ and
$f\left(  t_{0}+\rho\right)  =f\left(  t_{0}-\gamma\right)$.
Therefore,
$D_{\rho,\gamma}\left[  f\right]  \left(  t_{0}\right)  =0$.
\end{proof}

\begin{theorem}[The $\alpha,\beta$-symmetric Rolle mean value theorem]
\label{Symmetric Rolle's Mean Value Theorem}
Let $f:\left[  a,b\right] \rightarrow\mathbb{R}$
be a continuous function with $f\left(  a\right)  =f\left(  b\right)$.
Then there exist $\alpha$, $\beta\in\mathbb{R}^{+}$
and $c\in\left]  a,b\right[$ such that
$D_{\alpha,\beta}\left[  f\right]  \left(  c\right)  =0$.
\end{theorem}

\begin{proof}
If $f=const$, then the result is obvious.
If $f$ is not a constant function, then
there\ exists $t\in\left]a,b\right[$
such that $f\left(  t\right)  \neq f\left(  a\right)$.
Since $f$ is continuous on the compact set $\left[a,b\right]$,
$f$ has an extremum $M=f\left(  c\right)$
with $c\in\left]a,b\right[$. Since $c$ is also a local extremizer, then,
by Theorem~\ref{Symmetric Fermat's Theorem},
there exist $\alpha$,$\beta\in\mathbb{R}^{+}$ such that
$D_{\alpha,\beta}\left[  f\right]  \left(  c\right)  =0$.
\end{proof}

\begin{theorem}[The $\alpha,\beta$-symmetric Lagrange mean value theorem]
Let $f:\left[a,b\right]\rightarrow\mathbb{R}$ be a continuous function.
Then there exist $c\in\left]  a,b\right[$
and $\alpha$,$\beta\in\mathbb{R}^{+}$ such that
$D_{\alpha,\beta}\left[  f\right]  \left(  c\right)
=\frac{f\left(  b\right)-f\left(  a\right)  }{b-a}$.
\end{theorem}

\begin{proof}
Let function $g$ be defined on $\left[  a,b\right]$ by
$g\left(  t\right)  =f\left(  a\right)  -f\left(  t\right)  +\left(
t-a\right)  \frac{f\left(  b\right)  -f\left(  a\right)  }{b-a}$.
Clearly, $g$ is continuous on $\left[  a,b\right]$ and
$g\left(  a\right)  =g\left(  b\right)  =0$.
Hence, by Theorem~\ref{Symmetric Rolle's Mean Value Theorem}, there exist
$\alpha$, $\beta\in\mathbb{R}^{+}$ and $c\in\left]  a,b\right[$ such that
$D_{\alpha,\beta}\left[  g\right]  \left(  c\right)  =0$.
Since
\begin{align*}
D_{\alpha,\beta}\left[  g\right]  \left(  t\right)   & =\frac{g\left(
t+\alpha\right)  -g\left(  t-\beta\right)  }{\alpha+\beta}\\
& =\frac{1}{\alpha+\beta}\left(  f\left(  a\right)  -f\left(  t+\alpha\right)
+\left(  t+\alpha-a\right)  \frac{f\left(  b\right)  -f\left(  a\right)
}{b-a}\right) \\
& -\frac{1}{\alpha+\beta}\left(  f\left(  a\right)  -f\left(  t-\beta\right)
+\left(  t-\beta-a\right)  \frac{f\left(  b\right)  -f\left(  a\right)  }
{b-a}\right) \\
& =\frac{1}{\alpha+\beta}\left(  f\left(  t-\beta\right)  -f\left(
t+\alpha\right)  +\left(  \alpha+\beta\right)  \frac{f\left(  b\right)
-f\left(  a\right)  }{b-a}\right) \\
& =\frac{f\left(  b\right)  -f\left(  a\right)  }{b-a}-D_{\alpha,\beta}\left[
f\right]  \left(  t\right),
\end{align*}
we conclude that
$D_{\alpha,\beta}\left[  f\right]  \left(  c\right)  =\frac{f\left(  b\right)
-f\left(  a\right)  }{b-a}$.
\end{proof}

\begin{theorem}[The $\alpha,\beta$-symmetric Cauchy mean value theorem]
Let $f,g:\left[a,b\right] \rightarrow\mathbb{R}$ be continuous functions.
Suppose that $D_{\alpha,\beta}\left[  g\right]\left(  t\right)  \neq 0$
for all $t\in\left]  a,b\right[$ and all $\alpha$, $\beta\in\mathbb{R}^{+}$.
Then there exists $\bar{\alpha},\bar{\beta}\in\mathbb{R}^{+}$
and $c\in\left]  a,b\right[$ such that
$\frac{f\left(  b\right)  -f\left(  a\right)  }{g\left(  b\right)  -g\left(
a\right)  }=\frac{D_{\bar{\alpha},\bar{\beta}}\left[  f\right]  \left(
c\right)  }{D_{\bar{\alpha},\bar{\beta}}\left[  g\right]  \left(  c\right)}$.
\end{theorem}

\begin{proof}
From condition $D_{\alpha,\beta}\left[  g\right]  \left(  c\right)  \neq 0$
and the $\alpha,\beta$-symmetric Rolle mean value theorem
(Theorem~\ref{Symmetric Rolle's Mean Value Theorem}), it follows that
$g\left(  b\right)  \neq g\left(  a\right)$.
Let us consider function $F$ defined on $\left[a,b\right]$ by
$F\left(  t\right)  =f\left(  t\right)  -f\left(  a\right)  -\frac{f\left(
b\right)  -f\left(  a\right)  }{g\left(  b\right)  -g\left(  a\right)
}\left[  g\left(  t\right)  -g\left(  a\right)  \right]$.
Clearly, $F$ is continuous on $\left[  a,b\right]$
and $F\left(  a\right)  =F\left(b\right)$. Applying the
$\alpha,\beta$-symmetric Rolle mean value theorem to the
function $F$, we conclude that exist
$\bar{\alpha},\bar{\beta}\in\mathbb{R}^{+}$ and $c\in\left]a,b\right[$
such that
\begin{equation*}
0 = D_{\bar{\alpha},\bar{\beta}}\left[  F\right]  \left(  c\right)\\
=D_{\bar{\alpha},\bar{\beta}}\left[  f\right]  \left(  c\right)
-\frac{f\left(  b\right)-f\left(  a\right)}{g\left(  b\right)
-g\left(a\right)}D_{\bar{\alpha},\bar{\beta}}\left[g\right]\left(c\right),
\end{equation*}
proving the intended result.
\end{proof}


\begin{acknowledgement}
This work was supported by {\it FEDER} funds through
{\it COMPETE} --- Operational Programme Factors of Competitiveness
(``Programa Operacional Factores de Competitividade'')
and by Portuguese funds through the
{\it Center for Research and Development
in Mathematics and Applications} (University of Aveiro)
and the Portuguese Foundation for Science and Technology
(``FCT --- Funda\c{c}\~{a}o para a Ci\^{e}ncia e a Tecnologia''),
within project PEst-C/MAT/UI4106/2011
with COMPETE number FCOMP-01-0124-FEDER-022690.
Brito da Cruz was also supported by FCT through
the Ph.D. fellowship SFRH/BD/33634/2009.
\end{acknowledgement}




\begin{thebibliography}{99}

\bibitem{Agarwal:surv}
R. Agarwal, M. Bohner, D. O'Regan\ and\ A. Peterson,
Dynamic equations on time scales: a survey,
J. Comput. Appl. Math. {\bf 141} (2002), no.~1-2, 1--26.

\bibitem{MyID:188}
R. Almeida\ and\ D. F. M. Torres,
Nondifferentiable variational principles in terms of a quantum operator,
Math. Methods Appl. Sci. {\bf 34} (2011), no.~18, 2231--2241.
{\tt arXiv:1106.3831}

\bibitem{BohnerDEOTS}
M. Bohner\ and\ A. Peterson,
Dynamic equations on time scales,
Birkh\"auser Boston, Boston, MA, 2001.

\bibitem{withMiguel01}
A. M. C. Brito da Cruz, N. Martins\ and\ D. F. M. Torres,
Higher-order Hahn's quantum variational calculus,
Nonlinear Anal. {\bf 75} (2012), no.~3, 1147--1157.
{\tt arXiv:1101.3653}

\bibitem{MyID:111}
J. Cresson, G. S. F. Frederico\ and\ D. F. M. Torres,
Constants of motion for non-differentiable quantum variational problems,
Topol. Methods Nonlinear Anal. {\bf 33} (2009), no.~2, 217--231.
{\tt arXiv:0805.0720}

\bibitem{Ernst}
T. Ernst,
The different tongues of $q$-calculus,
Proc. Est. Acad. Sci. {\bf 57} (2008), no.~2, 81--99.

\bibitem{Jackson}
F. H. Jackson,
$q$-Difference Equations,
Amer. J. Math. {\bf 32} (1910), no.~4, 305--314.

\bibitem{Kac}
V. Kac\ and\ P. Cheung,
{\it Quantum calculus},
Universitext, Springer, New York, 2002.

\bibitem{Malinowska}
A. B. Malinowska\ and\ D. F. M. Torres,
On the diamond-alpha Riemann integral 
and mean value theorems on time scales,
Dynam. Systems Appl. {\bf 18} (2009), no.~3-4, 469--481.
{\tt arXiv:0804.4420}

\bibitem{MalinowskaTorres}
A. B. Malinowska\ and\ D. F. M. Torres,
The Hahn quantum variational calculus,
J. Optim. Theory Appl. {\bf 147} (2010), no.~3, 419--442.
{\tt arXiv:1006.3765}

\bibitem{Milne}
L. M. Milne-Thomson,
{\it The calculus of finite differences},
Macmillan and Co., Ltd., London, 1951.

\bibitem{Thomsom}
B. S. Thomson,
{\it Symmetric properties of real functions},
Monographs and Textbooks in Pure and Applied Mathematics,
183, Dekker, New York, 1994.

\end{thebibliography}
\end{document}